\documentclass[12pt]{amsart}
\usepackage{amsfonts}
\usepackage{amscd}
\usepackage{amssymb}
\usepackage{graphics}

\usepackage{amsmath}

\usepackage{hyperref}
\hypersetup{colorlinks,citecolor=blue}

\newtheorem{theorem}{Theorem}

\newtheorem{corollary}[theorem]{Corollary}

\newtheorem{definition}[theorem]{Definition}

\newtheorem{lemma}[theorem]{Lemma}

\newtheorem{remark}[theorem]{Remark}

\newcommand{\BN}{{\mathbb{N}}}

\newcommand{\Cl}{\mathrm{Cl}}

\newcommand{\zz}{\mathbb{Z}[\frac{1}{2}]}

\theoremstyle{definition}

\catcode`\@=11
\long\def\@savemarbox#1#2{\global\setbox#1\vtop{\hsize\marginparwidth 
  \@parboxrestore\tiny\raggedright #2}}
\catcode`\@=12

\begin{document}
\author{Tsachik Gelander}
\address{Hebrew University}
\email{gelander@math.huji.ac.il}
\author{Gili Golan}
\address{Vanderbilt University}
\email{gili.golan@vanderbilt.edu}
\author{Kate Juschenko}
\address{Northwestern University}
\email{kate.juschenko@gmail.com}
\thanks{The work of the first author was partially supported by ISF-Moked grant 2095/15. The work of the second author was partially supported by a Fulbright grant and a postdoctoral scholarship from Bar Ilan University. The work of the third author was partially supported by NSF CAREER grant DMS 1352173.}

\title{Invariable generation of Thompson groups}
\begin{abstract}
A subset $S$ of a group $G$ {\it invariably generates} $G$ if $G= \langle s^{g(s)} | s \in S\rangle$ for every choice of $g(s) \in G,s \in S$. We say that a group $G$ is {\it invariably generated} if such $S$ exists, or equivalently if $S=G$ invariably generates $G$. In this paper, we study invariable generation of Thompson groups. We show that Thompson group $F$ is invariable generated by a finite set, whereas Thompson groups $T$ and $V$ are not invariable generated.
\end{abstract}

\maketitle

\section{Introduction} 
Recall that a subset $S$ of a group $G$ {\it invariably generates} $G$ if $G= \langle s^{g(s)} | s \in S\rangle$ for every choice of $g(s) \in G,s \in S$. One says that a group  $G$ is {\it invariably generated}, or shortly IG, if such $S$ exists, or equivalently if $S=G$ invariably generates $G$. 
The term ``invariable generation'' was coined by Dixon \cite{Dixon} in his study of generation of Galois groups, where elements are given only up to conjugacy. Invariably generated groups were studied before by Wiegold \cite{W76,W77} under different terminology: A group $G$ is invariably generated if and only if no proper subgroup of $G$ meets every conjugacy class. This is equivalent to saying that every transitive permutation representation of $G$ on a non-singleton set admits a fixed-point-free element. Following \cite{KLS}, we say that $G$ is {\it finitely invariably generated}, or shortly FIG, if there is a finite subset $S\subset G$ which invariably generates $G$. 

A simple counting argument shows that every finite group is IG. Obviously, abelian groups are IG. More generally, J. Wiegold \cite{W76} showed that the class of IG groups is closed under extensions, hence contains all virtually solvable groups. Clearly, this class is also closed to quotients.
On the other hand, the class of IG groups is not closed under direct unions; for instance the (locally finite) group of finitely supported permutations of $\BN$ is clearly not IG, since every element fixes some point in $\BN$. Moreover, Wiegold \cite{W77} gave an example of an IG group whose commutator subgroup is not IG, proving in particular that the class IG is not subgroup closed.

In \cite{W76}, Wiegold proved that the free group 
$F=\langle a,b\rangle$ is not IG by producing a list $L$ of conjugacy class representatives which are jointly independent (i.e., that freely generate a free group of infinite rank). To recall his construction, let $\{ w_n\}$ be conjugacy class representatives which start and end with a non-zero power of $b$, then take $L=\{w_n^{a^n}:n\in\BN\}$. T. Gelander proved in \cite{convergence} that convergence groups, and in particular Gromov hyperbolic groups and relatively hyperbolic groups, are not IG, confirming a conjecture from \cite{KLS}. Gelander and Meiri \cite{Chen} established various examples of arithmetic groups possessing the Congruence Subgroup Property which are not IG, providing a negative answer to a question from \cite{KLS}. 

The notion of finite invariant generation is more subtle. For example it is still unknown weather every finitely generated solvable group is FIG.  Kantor, Lubotzky and Shalev \cite{KLS} proved that a finitely generated linear group is finitely invariably generated if and only if it is virtually solvable.

The main result of the paper is summarized in the following theorem:

\begin{theorem}
 Thompson's group $F$ is finitely invariably generated. The Thompson groups $T$ and $V$ are not invariably generated. 
\end{theorem}

\vskip .2cm

\section{Thompson group F}\label{s:FT}

\subsection{F as a group of homeomorphisms}

Recall that Thompson group $F$ is the group of all piecewise linear homeomorphisms of the interval $[0,1]$ with finitely many breakpoints where all breakpoints are finite dyadic and all slopes are integer powers of $2$.  
The group $F$ is generated by two functions $x_0$ and $x_1$ defined as follows \cite{CFP}.
	
	\[
   x_0(t) =
  \begin{cases}
   2t &  \hbox{ if }  0\le t\le \frac{1}{4} \\
   t+\frac14       & \hbox{ if } \frac14\le t\le \frac12 \\
   \frac{t}{2}+\frac12       & \hbox{ if } \frac12\le t\le 1
  \end{cases} 	\qquad	
   x_1(t) =
  \begin{cases}
   t &  \hbox{ if } 0\le t\le \frac12 \\
   2t-\frac12       & \hbox{ if } \frac12\le t\le \frac{5}{8} \\
   t+\frac18       & \hbox{ if } \frac{5}{8}\le t\le \frac34 \\
   \frac{t}{2}+\frac12       & \hbox{ if } \frac34\le t\le 1 	
  \end{cases}
\]

The composition in $F$ is from left to right.

Every element of $F$ is completely determined by how it acts on the set $\zz$. Every number in $(0,1)$ can be described as $.s$ where $s$ is an infinite word in $\{0,1\}$. For each element $g\in F$ there exists a finite collection of pairs of (finite) words $(u_i,v_i)$ in the alphabet $\{0,1\}$ such that every infinite word in $\{0,1\}$ starts with exactly one of the $u_i$'s. The action of $F$ on a number $.s$ is the following: if $s$ starts with $u_i$, we replace $u_i$ by $v_i$. For example, $x_0$ and $x_1$  are the following functions:

\[
   x_0(t) =
  \begin{cases}
   .0\alpha &  \hbox{ if }  t=.00\alpha \\
    .10\alpha       & \hbox{ if } t=.01\alpha\\
   .11\alpha       & \hbox{ if } t=.1\alpha\
  \end{cases} 	\qquad	
   x_1(t) =
  \begin{cases}
   .0\alpha &  \hbox{ if } t=.0\alpha\\
   .10\alpha  &   \hbox{ if } t=.100\alpha\\
   .110\alpha            &  \hbox{ if } t=.101\alpha\\
   .111\alpha                      & \hbox{ if } t=.11\alpha\
  \end{cases}
\]
where $\alpha$ is any infinite binary word.

\subsection{Elements of F as pairs of binary trees} \label{sec:red}

Often, it is more convenient to describe elements of $F$ using pairs of finite binary trees (see \cite{CFP} for a detailed exposition). The considered binary trees are rooted \emph{full} binary trees; that is, every inner vertex (i.e., non-leaf vertex) has two outgoing edges: a left edge and a right edge. A  \emph{branch} in a binary tree is a simple path from the root to a leaf. If every left edge in the tree is labeled ``0'' and every right edge is labeled ``1'', then a branch in $T$ has a natural binary label. We rarely distinguish between a branch and its label. 

Let $(T_+,T_-)$ be a pair of finite binary trees with the same number of leaves. $(T_+,T_-)$ is called a \emph{tree-diagram}. Let $u_1,\dots,u_n$ be the (labels of) branches in $T_+$, listed from left to right. Let $v_1,\dots,v_n$ be the (labels of) branches in $T_-$, listed from left to right. We say that the tree-diagram $(T_+,T_-)$ has the \emph{pair of branches} $u_i\rightarrow v_i$ for $i=1,\dots,n$. The tree-diagram $(T_+,T_-)$ \emph{represents} the function $g\in F$ which takes binary fraction $.u_i\alpha$ to $.v_i\alpha$ for every $i$ and every infinite binary word $\alpha$. We also say that the element $g$ takes the branch $u_i$ to the branch $v_i$.
For a finite binary word $u$, we denote by $[u]$ (resp. $(u]$) the dyadic interval $[.u,.u1^{\mathbb{N}}]$ (resp. $(.u,.u1^{\mathbb{N}}]$). If $u\rightarrow v$ is a pair of branches of $(T_+,T_-)$, then $g$ maps the interval $[u]$ linearly onto $[v]$. 

A \emph{caret} is a binary tree composed of a root with two children. If $(T_+,T_-)$ is a tree-diagram and one attaches a caret to the $i^{th}$ leaf of $T_+$ and the $i^{th}$ leaf of $T_-$ then the resulting tree diagram is \emph{equivalent} to $(T_+,T_-)$ and represents the same function in $F$. When we say that a function $f$ has a pair of branches $u_i\rightarrow v_i$, the meaning is that some tree-diagram representing $f$ has this pair of branches. %In other words, this is equivalent to saying that $f$ maps the dyadic interval $[u_i]$ linearly onto $[v_i]$.
 Clearly, if $u\rightarrow v$ is a pair of branches of $f$, then for any finite binary word $w$, $uw\rightarrow vw$ is also a pair of branches of $f$. Similarly, if $f$ has the pair of branches $u\rightarrow v$ and $g$ has the pair of branches $v\rightarrow w$ then $fg$ has the pair of branches $u\rightarrow w$.

\subsection{Generating sets of F}

Let $H\le F$. Following \cite{GS,G16}, we define the \emph{closure} of $H$, denote $\Cl(H)$, to be the subgroup of $F$ of all piecewise-$H$ functions. In other words, $\Cl(H)$ is the topological full group (see, for example, \cite{EM}) of the group $H$ acting on the set of finite dyadic fractions in the unit interval $[0,1]$ with the natural topology. In \cite{G16}, the first author proved that the generation problem in $F$ is decidable. That is, there is an algorithm that decides given a finite subset $X$ of $F$ whether it generates the whole $F$. 

\begin{theorem}\cite[Theorem 7.14]{G16}\label{gen}
Let $H$ be a subgroup of $F$. Then $H=F$ if and only if the following conditions are satisfied. 
\begin{enumerate}
\item[(1)] $\Cl(H)=F$. 
\item[(2)] $H[F,F]=F$.
\item[(3)] There is an element $h\in H$ which fixes a finite dyadic fraction $\alpha\in (0,1)$ such that $h'(\alpha^-)=1$ and 
$h'(\alpha^+)=2$.
\end{enumerate}
\end{theorem}

Below we would apply Theorem \ref{gen} to prove that a given subset of $F$ is a generating set of $F$. To verify that Condition (1) in the theorem holds, we will make use of the following lemma. The lemma is an immediate result of Remark 7.2 and Lemma 10.6 in \cite{G16}.

\begin{lemma}\label{suffice}
Let $H$ be a subgroup of $F$. If for each of the following pairs of branches there is an element in $H$ which has the given pair of branches, then $\Cl(H)=F$. 
\begin{enumerate}
\item[(1)] $00\rightarrow 0$
\item[(2)] $1\rightarrow 11$
\item[(3)] $01\rightarrow 10$
\item[(4)] $010\rightarrow 10$. 
\item[(5)] $011\rightarrow 10$. 
\end{enumerate}
\end{lemma}

\subsection{Thompson groups T and V}

Consider the circle $S^1$ as the unit interval $[0,1]$ with endpoints $0$ and $1$ identified. In particular, 
we can refer to dyadic fractions and subintervals of $S^1$ (including subintervals of the form $[a,b]$ for $a>b$ which contain the point $0=1$). Thompson group $T$ is the group of all piecewise-linear orientation preserving homeomorphisms of $S^1$ which preserve the set of finite dyadic fractions, have finitely many breakpoints, all of which are at finite dyadic fractions and where all slopes are integer powers of $2$. Thompson group $F$ can be viewed as the subgroup of $T$ which fixes $0$. 

Thompson group $V$ is the group of left-continuous bijections of $S^1$ which map finite dyadic fractions to finite dyadic fractions, that are differentiable except at finitely many finite dyadic fractions and such that on each maximal interval where the functions are differentiable, they are linear with slope an integer power of $2$. Thompson group $T$ is clearly a subgroup of $V$. 

Elements of $T$ and $V$ can be represented by tree-diagrams. Let $T_+,T_-$ be a pair of trees with $n$ leaves each. Let $\sigma\in S_n$ be a permutation of $\{1,\dots,n\}$. The triple $(T_+,\sigma,T_-)$ represents an element of $V$ as follows. Let $u_1,\dots,u_n$ (resp. $v_1,\dots,v_n$) be the branches of $T_+$ (resp., $T_-$), listed from left to right. The function $g\in V$ \emph{represented} by $(T_+,\sigma,T_-)$ maps $.u_i\alpha$ to $.v_{\sigma(i)}\alpha$ for every $i$ and every infinite binary word $\alpha\not\equiv 0^{\mathbb{N}}$. In other words, $g$ maps the interval $(u_i]=(.u_i,.u_i1^{\mathbb{N}}]$ linearly onto the interval $(v_{\sigma(i)}]=(.v_{\sigma(i)},.v_{\sigma(i)}1^{\mathbb{N}}]$. The triple $(T_+,\sigma,T_-)$ is called a \emph{tree-diagram} of $g$. As in the case of tree diagrams of elements of $F$, one can define equivalent tree-diagrams (which represent the same element of $V$) and are obtained by insertion (or reduction) of common carets. A tree-diagram $(T_+,\sigma,T_-)$ with $\sigma\in S_n$ represents an element of Thompson group $T$ if and only if $\sigma$ is a power of the cycle $(1\ 2\ \dots\ n)$ (that is, if and only if $\sigma$ preserves the cyclic order of $\{1,\dots,n\}$).

Let $(T_+,\sigma,T_-)$ be a tree-diagram of an element $g$ in $V$. We can consider $T_+$ and $T_-$ as rooted subtrees of the complete infinite binary tree $T_{\infty}$. If $T_+$ and $T_-$ coincide; i.e., have the same set of branches $\{u_1,\dots,u_n\}$, then $\sigma$ can be viewed as a permutation on the set $\{u_1,\dots,u_n\}$. In that case, it is easy to see that $g$ is periodic of the same order as $\sigma$. If $T_+$ and $T_-$ do not coincide, then since $T_+$ and $T_-$ have the same number of carets, $T_+$ must have a branch $v$ such that $v$ is a strict prefix of some branches $vw_1,\dots,vw_m$  of $T_-$. Clearly, $m\ge 2$. We shall need the following lemma which follows immediately from \cite[Lemma 10.5]{Brin}.

\begin{lemma}\label{Brin_lem}
Let $g\in V$ be a non periodic element. Then $g$ is represented by a tree-diagram $(T_+,\sigma,T_-)$ such that the following holds. If $v$ is a branch of $T_+$ and $v$ is a proper prefix of branches $vw_1,\dots,vw_m$ of $T_-$ (for $m\ge 2$) then there is $k\in\{1,\dots,m\}$ and $r\in\mathbb{N}$ such that the intervals $(v],g(v],\dots,g^{r-1}(v]$ are pairwise disjoint and such that $g^r(v]=(vw_k]$.
\end{lemma}

We will also need the following two remarks. Remark \ref{irrational} can be proved in a similar way to \cite[Proposition 3.1]{Sav} (see also \cite[Corollary 2.5]{GS2}). Remark \ref{transitive} follows easily from \cite[Lemma 4.2]{CFP}.

\begin{remark}\label{irrational}
Let $g\in V$ and assume that $g$ fixes an irrational number $\alpha$. Then $g$ fixes pointwise a small neighborhhod $(\alpha-\epsilon, \alpha+\epsilon)$ of $\alpha$.
\end{remark}

\begin{remark}\label{transitive}
Let $(a,b)$ and $(c,d)$ be open sub-intervals of $S^1$ with dyadic endpoints. Then there is an element $g\in T$ (and in particular, $g\in V$) which maps $(a,b)$ onto $(c,d)$. 
\end{remark}

\section{Thompson group F is invariably generated}

Recall that $x_0,x_1$ are the generators of $F$ defined above. In this section we prove that Thompson group $F$ is invariably generated by $\{x_0,x_1,x_0x_1\}$. We would need the following two lemmas. 

\begin{lemma}\label{conj}
Let $g\in F$. Let $H=\langle (x_0x_1)^g\rangle $. Then there is an element $f\in H$ such that for some $m,n\in\mathbb{N}$ the function $f$ has the pairs of branches 
\begin{enumerate}
\item[(1)] $0^m10\rightarrow 1^n0$. %for some $m,n\in\mathbb{N}$. 
\item[(2)] $0^m11\rightarrow 1^{n+1}0$. %for some $k,l\in\mathbb{N}$. 
\end{enumerate}
\end{lemma}

\begin{proof}
The function $g$ has pairs of branches $0^a\rightarrow 0^b$ and $1^c\rightarrow 1^d$ for some $a,b,c,d\in\mathbb{N}$.
We observe that $x_0x_1$ is composed of the following pairs of branches. 
\[
  \begin{cases}
	00 & \rightarrow 0\\
	010  & \rightarrow 10\\
	011 & \rightarrow 110\\
	1 & \rightarrow 111\
  \end{cases}
\]

Let $h=x_0x_1$. We note that the pairs of branches $00\rightarrow 0$ and $010\rightarrow 10$ of $h$ imply that $h^a$ has the pair of branches $0^a10\rightarrow 10$. Indeed, if $a=1$, this is clear. If $a>1$, then $h$ takes $0^a10$ to $0^{a-1}10$ and we are done by induction. Similarly, the pairs of branches $00\rightarrow 0$ and $011\rightarrow 110$ of $h$ imply that $h^a$ has the pair of branches $0^a11\rightarrow 110$.
The pair of branches $1\rightarrow 111$  of $h$ implies that $h^c$ has the pairs of branches $10\rightarrow 1^{2c+1}0$ and $110\rightarrow 1^{2c+2}0$. Hence $h^{a+c}$ has the pairs of branches $0^a10\rightarrow 1^{2c+1}0$ and $0^a11\rightarrow 1^{2c+2}0$. 

Let $f=(h^{a+c})^g$. Then $f$ has the pairs of branches 
$$0^b10\rightarrow 1^{c+d+1}0\ \mbox{ and }\  0^b11\rightarrow 1^{c+d+2}0.$$
Indeed, $g^{-1}$ takes $0^b10$ to $0^a10$, then $h^{a+c}$ takes $0^a10$ to $1^{2c+1}0$. Then $g$ takes $1^c1^{c+1}0$ to $1^d1^{c+1}0$. That gives the first pair of branches. The proof for the second one is similar.
Therefore, $f\in H$ is an element as described.  
\end{proof}

\begin{lemma}\label{prop}
Let $g\in F$. Let $H=\langle x_0,(x_0x_1)^g\rangle$. Then $\Cl(H)=F$.
\end{lemma}

\begin{proof}
The element $x_0$ is composed of the following pairs of branches. 
\[
  x_0=
  \begin{cases}
  00 &  \rightarrow		0\\
  01 &  \rightarrow 10\\
	1 &  \rightarrow 11\
	  \end{cases} 
\]
Since $x_0\in H$, by Lemma \ref{suffice}, it suffices to prove that there is an element in $\Cl(H)$ with the pair of branches $010\rightarrow 10$ and an element in $H$ with the pair of branches $011\rightarrow 10$. 

By Lemma \ref{conj}, there is an element $f\in H$ with pairs of branches 
$0^m10\rightarrow 1^n0$ and $0^m11\rightarrow 1^{n+1}0$ for some $m,n\in\mathbb{N}$.
Let 
$$h_1=x_0^{-(m-1)} f x_0^{-(n-1)} \ \mbox{ and }\ h_2=x_0^{-(m-1)} f x_0^{-n}.$$
Then $h_1,h_2\in H$. The function $h_1$ has the pair of branches $010\rightarrow 10$. Indeed, $x_0^{-(m-1)}$ has the pair of branches $010\rightarrow 0^{m}10$, then $f$ has the pair of branches $0^m10\rightarrow 1^n0$, then $x_0^{-(n-1)}$ has the pair of branches $1^n0\rightarrow 10$. Similarly, $h_2$ has the pair of branches 
$011\rightarrow 10$. Therefore, $\Cl(H)=F$. 
\end{proof}

\begin{theorem}\label{thm:F}
Thompson group $F$ is invariably generated by $\{x_0,x_1,x_0x_1\}$. 
\end{theorem}

\begin{proof}
It suffices to prove that for any $h,g\in F$, the set $X=\{x_0,x_1^h,(x_0x_1)^g\}$ is a generating set of $F$. Let $H$ be the subgroup of $F$ generated by $X$. Then $H[F,F]=F$ (indeed, the image of $X$ in the abelianization of $F$ contains the image of the generating set $\{x_0,x_1\}$). 
By Lemma \ref{prop}, $\Cl(H)=F$. In addition, since $x_1$ fixes $\frac{1}{2}$ and has slope $1$ at $\frac{1}{2}$ on the left and slope $2$ on the right, the element $x_1^h$ fixes the finite dyadic fraction $\alpha=h(\frac{1}{2})$ and has slope $1$ at $\alpha$ on the left and slope $2$ on the right. Hence, by Theorem \ref{gen}, we have $H=F$. 
\end{proof}

\section{Thompson groups T and V are not invariably generated}

To prove that $T$ and $V$ are not invariably generated we will need the definition of a $\gamma$-wandering set. 

\begin{definition}
Let $\gamma$ be an element of $V$. A subset $U\subseteq S^1$ is \emph{$\gamma$-wandering} if for every $n\in\mathbb{Z}$ with $\gamma^n\neq e$ we have $\gamma^n(U)\cap U=\emptyset$. The set $U$ is \emph{weakly $\gamma$-wandering} if for every $n\in\mathbb{Z}$, either $\gamma^n$ fixes $U$ pointwise or $\gamma^n(U)$ is disjoint from $U$. 
\end{definition}

\begin{lemma}\label{lem1}
Let $\gamma\in V$ be a non-periodic element. Then there is an open $\gamma$-wandering set. 
\end{lemma}

\begin{proof}
Let $(T_+,\sigma,T_-)$ be a tree-diagram of $\gamma$ which satisfies the conditions of Lemma \ref{Brin_lem}. By the discussion preceding that lemma, $T_+$ has a branch $v$ which is a proper prefix of some branches $vw_1,\dots,vw_m$ of $T_-$, where $m\ge 2$. Then, by Lemma \ref{Brin_lem}, there is $k\in\{1,\dots,m\}$ and $r\in\mathbb{N}$ such that the intervals $(v],g(v],\dots,g^{r-1}(v]$ are pairwise disjoint and such that $g^r(v]=(vw_k]$. Let $j\in\{1,\dots,m\}$ be distinct from $k$. We claim that $(vw_j]$ is $\gamma$-wandering (and as such, every open interval contained in $(vw_j]$ is also $\gamma$-wandering). Indeed, let $n\in\mathbb{Z}\setminus\{0\}$. If $\gamma^n(U)\cap U\neq\emptyset$ then $U\cap \gamma^{-n}(U)\neq\emptyset$. Hence, we can assume that $n$ is positive. We write $n=ar+b$ where $a\ge 0$ and $b\in\{0,\dots,r-1\}$. Since $\gamma^r(v]=(vw_k]\subseteq(v]$, we have $\gamma^{ar}(v]\subseteq(vw_k]$. Hence $\gamma^{ar}(vw_j]\subseteq(vw_k]$. If $b=0$ we are done since $j\neq k$ implies that $(vw_j]\cap(vw_k]=\emptyset$. If $b\neq 0$ then $\gamma^{ar+b}(vw_j]\subseteq\gamma^b(v]$. Since $b\in\{1,\dots,r-1\}$, $\gamma^b(v]$ is disjoint from $(v]$. In particular, $\gamma^{ar+b}(vw_j]$ is disjoint from $(vw_j]$. 
\end{proof}

\begin{lemma}\label{lem2}
Let $\gamma$ be a periodic element of $V$. Then $\gamma$ admits an open weakly $\gamma$-wandering set. If $\gamma\in T$ then $\gamma$ admits an open $\gamma$-wandering set (in fact, in that case, any weakly $\gamma$-wandering set is $\gamma$-wandering). 
\end{lemma}

\begin{proof}
Assume that $\gamma\in V$ has period $n$. We can assume that $\gamma\neq e$ (equivalently, that $n>1$). Let $b$ be an irrational number in $S^1$ such that $b$ is not fixed by $\gamma$. Let $m$ be the period of $b$ under the action of $\langle \gamma\rangle$. Since $\gamma,\dots,\gamma^{m-1}$ are left-continuous and do not fix $b$, for a small enough interval $(a,b]$, the intervals $\gamma(a,b],\dots,\gamma^{m-1}(a,b]$ are disjoint from $(a,b]$. Since $\gamma^m$ fixes $b$ and $b$ is irrational, by Remark \ref{transitive}, $\gamma^m$ fixes pointwise an open neighborhood of $b$. In particular, we can assume that $(a,b]$ is fixed pointwise by $\gamma^m$. Then the interval $(a,b]$ (and as such, the interval $(a,b)$) is weakly $\gamma$-wandering. 

Now assume that $\gamma\in T$. Then the period $m$ of $b$ is equal to the order $n$ of $\gamma$. Indeed, $\gamma$ being periodic means that every $x\in S^1$ has a finite orbit under the action of $\gamma$. But an orientation preserving homeomorphism of $S^1$ cannot have two finite orbits of different periods (see, for example, \cite{KH}). Thus, every orbit of $\langle\gamma\rangle$ is of period $m$. That implies that the order of $\gamma$ is $m$. As noted above, $\gamma(a,b],\dots,\gamma^{m-1}(a,b]$ are disjoint from $(a,b]$. Hence, $(a,b]$ is a $\gamma$-wandering set. 
\end{proof}

Lemmas \ref{lem1} and \ref{lem2} imply the following corollaries. The proof of Corollary \ref{cor2} is almost identical to the proof of Corollary \ref{cor1}.

\begin{corollary}\label{cor1}
Let $\gamma\in V$ and let $A\subset S^1$ be a closed subset of $S^1$ such that $A\neq S^1$. Then there exists $g\in V$ such that $A$ is weakly $\gamma^g$-wandering. 
\end{corollary}

\begin{proof}
By Lemmas \ref{lem1} and \ref{lem2} there is an open weakly $\gamma$-wandering set $U\subseteq S^1$. We can assume that $U=(a,b)$ is an open subinterval of $S^1$ with dyadic endpoints. Let $(c,d)$ be an open subinterval of $S^1$ with dyadic endpoints such that $(c,d)\supseteq A$. It follows from Remark \ref{transitive} that there is an element $g\in T$ such that $g(a,b)=(c,d)\supseteq A$. Since $g(a,b)$ is weakly $\gamma^g$-wandering, the set $A$ is also weakly $\gamma^g$-wandering. 
\end{proof}

\begin{corollary}\label{cor2}
Let $\gamma\in T$ and let $A\subset S^1$ be a closed subset of $S^1$ such that $A\neq S^1$. Then there exists $g\in V$ such that $A$ is $\gamma^g$-wandering. 
\end{corollary}

Now, we are ready to prove that $T$ and $V$ are not invariably generated. We start with the proof for $T$ and then adapt it to a proof for Thompson group $V$. 

\begin{theorem}
Thompson group $T$ is not invariably generated.
\end{theorem}

\begin{proof}
Let $I_n$ be a set of pairwise disjoint open subintervals  of $S^1$. %converging to a point.
Let $C_n$ be the non-trivial conjugacy classes of $T$. Let $\gamma_n\in C_n$ be such that the complement $I_n^c$ is $\gamma_n$-wandering ($\gamma_n$ exists by Corollary \ref{cor2}). Then for any $k\in\mathbb{Z}$, if $\gamma_n^k\neq e$ then $\gamma_n^k(I_n^c)\subseteq I_n$. In particular, for all $i\neq n$, $\gamma_n^k(I_i)\subseteq I_n$. Then by the ping-pong lemma (see, for example, \cite[Lemma 2.1]{OS})  $\langle \gamma_n: n\in \mathbb{N}\rangle=\ast_{n\in\mathbb{N}} \langle \gamma_n\rangle$.
\end{proof}

\begin{theorem}
Thompson group $V$ is not invariably generated. 
\end{theorem}

\begin{proof}
Let $I_n$ be a set of pairwise disjoint open intervals converging to the point $0\in S^1$. We can assume that $I_n\subseteq (0,\frac{1}{4})$ for all $n$. Let $C_n$ be the non-trivial conjugacy classes of $V$. Let $\gamma_n\in C_n$ be such that the complement $I_n^c$ is weakly $\gamma_n$-wandering ($\gamma_n$ exists by Corollary \ref{cor1}). We claim that the orbit of $0$ under the action of $H=\langle \gamma_n : n\in \mathbb{N}\rangle$ is contained in $[0,\frac{1}{4})$. Since $V$ acts transitively on the set of finite dyadic fractions in $S^1$, that would imply that $H\neq V$, as required. Thus, the following lemma completes the proof of the theorem. 

\begin{lemma}
Let $\alpha\neq 0$ be a point in the orbit of $0$ in $S^1$ under the action of $H$. Let $g\in H$ be an element of minimal word-length over the alphabet $\{ \gamma_n : n\in \mathbb{N}\}$ such that $g(0)=\alpha$. Assume that 
$$g=\gamma_{i_1}^{k_1}\cdots \gamma_{i_m}^{k_m}$$ 
where $k_1,\dots,k_m\neq 0$; for each $j=1,\dots,m-1$, ${i_{j+1}}\neq {i_j}$; and such that $|k_1|+\dots+|k_m|$ is the word length of $g$ over $\{ \gamma_n : n\in \mathbb{N}\}$. Then $\alpha\in I_{i_m}$. 
\end{lemma}

\begin{proof}
We prove the lemma by induction on $m$. If $m=1$ then $g=\gamma_{i_1}^{k_1}$. The set $I_{i_1}^c$ is weakly $\gamma_{i_1}$-wandering. Since $\gamma_{i_1}^{k_1}(0)=\alpha\neq 0$ and $0\in I_{i_1}^c$, the set $\gamma_{i_1}^{k_1}(I_{i_1}^c)$ is disjoint from $I_{i_1}^c$; i.e., $g(I_{i_1}^c)=\gamma_{i_1}^{k_1}(I_{i_1}^c)\subseteq I_{i_1}$. Hence $\alpha\in I_{i_1}$. Assume that the lemma holds for $m$ and let 
$$g=\gamma_{i_1}^{k_1}\cdots \gamma_{i_m}^{k_m}\gamma_{i_{m+1}}^{k_{m+1}}.$$
We let $g_1=\gamma_{i_1}^{k_1}\cdots \gamma_{i_m}^{k_m}$ and $\beta=g_1(0)$. By the induction hypothesis, $\beta\in I_{i_m}$. Since $i_m\neq i_{m+1}$, $\beta\in I_{i_m}\subseteq I_{i_{m+1}}^c$. By the assumption of minimality of the word length of $g$, $\alpha=g(0)=\gamma_{i_{m+1}}^{k_{m+1}}(\beta)\neq \beta$. Hence, $\gamma_{i_{m+1}}^{k_{m+1}}$ does not fix pointwise the interval $I_{i_{m+1}}^c$. Since $I_{i_{m+1}}^c$ is weakly $\gamma_{i_{m+1}}$-wandering, that implies that $\gamma_{i_{m+1}}^{k_{m+1}}(I_{i_{m+1}}^c)\subseteq I_{i_{m+1}}$. Hence $\alpha\in I_{i_{m+1}}$ as required. 
\end{proof}

\end{proof}

\end{document}